\documentclass[12pt,a4paper,final]{article}
\usepackage[utf8]{inputenc}
\usepackage[english]{babel}
\usepackage{amsmath,amsthm,amscd}
\usepackage{amsfonts}
\usepackage{amssymb}
\usepackage{makeidx}
\usepackage{mathrsfs}
\usepackage{float}
\usepackage[pdftex]{graphicx,color}
\usepackage{bmpsize}
\usepackage{url}

\newtheorem{Th}{Theorem}[section]
\newtheorem{Cor}[Th]{Corollary}
\newtheorem{Lem}[Th]{Lemma}
\newtheorem{Prop}[Th]{Proposition}

\theoremstyle{definition}
\newtheorem{Def}[Th]{Definition}

\newtheorem{Rem}[Th]{Remark}

\usepackage[left=2cm,right=2cm,top=2cm,bottom=2cm]{geometry}
\title{\textbf{On Extended Family of Fractional Sobolev Spaces Via Ultradistributions of Slow Growth}}
\author{\textbf{$^1$Anslem Amaonyeiro,\quad $^{2*}$Egwe Murphy E.}\\
\small $^{1,2^*}$Department of Mathematics, University of Ibadan, P. M. B. 200132, Ibadan, Nigeria.\\
\small  anslemamaonyeiro@uam.edu.ng; $^*$murphy.egwe@ui.edu.ng, egwemurphy@gmail.com
       }
       \date{}
\begin{document}
\bibliographystyle{unsrt}
\addto\captionsenglish{
\renewcommand{\bibname}{\centering REFERENCES}}
\maketitle
\ \\
\begin{abstract}
This paper considers a new version of fractional Sobolev spaces $\widetilde{\mathcal{W}}_{\mathcal{U}}^{\beta,p}(\mathbb{C}^{n})$ defined using the concept of tempered ultradistributions with respect to the spaces of ultradifferentiable functions $\mathcal{U}(\mathbb{C})$. The space $\widetilde{\mathcal{W}}_{\mathcal{U}}^{\beta,p}(\mathbb{C}^{n})$ is a natural generalization of the classical Sobolev space with integer order, where some additional conditions of growth control have been introduced. We analyze some possible definitions and their roles in the structure theory. We prove some density and compact embedding results, investigating the possibility of the extension domains. Some of the results we present here are extensions of the existing ones with some additional conditions. The construction of the new family of fractional Sobolev space is considered within the framework of Fourier transform of ultradistributions of slow growth.
\end{abstract}
\textbf{Keyword}: Tempered ultradistributions. Fractional Sobolev spaces. Density. Extension. Embedding\ \\
\textbf{MSC Classification:}46F10, 46E35, 46F05, 34K37, 35R11, 26A33

\section{Introduction}

\noindent The theories of the respective foundations of fractional Sobolev spaces and spaces of ultradistributions of slow growth have been known for many years forming the cornerstone for this work and provide different important settings for studying the notions of boundary values problems of the partial differential equations (PDEs). The notion of Tempered ultradistributions is more fundamentally more general than the notion of tempered distributions which in turn form an extension in the notion of generalized functions through generalized transform. Many researchers have contributed to the emergence of tempered ultradistributions such as Al-Omari and J. F. M. Al-Omari \cite{1}, Pathak \cite{9}, and others in \cite{7,8}. These authors gave some impressive definitions to the various spaces of tempered ultradistributions, $\mathcal{U}'$ as duals of ultradifferentiable functions of rapid decrease to extend the Schwartz space $\mathcal{S}'$ of tempered distributions \cite{12}. \ \\
\nonumber More works involving the notions of fractional Sobolev spaces \cite{11,10} have been carried out.  Also a research interest question arises when one asks about fractional ultra-derivatives of tempered ultradistributions rather than the existing functions in the fractional Sobolev spaces. Motivated by such valid question the authors seek for a natural generalization by the use of tempered ultradistributions, which is more general than tempered distributions and seeks to generalize fractional derivatives with the inclusion of slow growth condition where the supremum of such function product with ultra-polynomial is finite. The theory of fractional Sobolev spaces of tempered ultradistributionstheory will serve as a unifying concept in the spirit of the natural generalization of full derivative and the slow growth conditions. \ \\
\noindent The subsequent parts of this work is organized as follows. In the preliminary section 2 we recall some notions from the theories of fractional Sobolev spaces and spaces of tempered ultradistributions including some existing definitions of the fractional Sobolev spaces. In the section we also present some characterizations relevant to the two theories and their corresponding properties. In section 3 we introduce the concept of the new fractional Sobolev space using both the underlying topologies of the both notions. In this section we also present some properties attributed to the newly constructed space as analogous to the integer order Sobolev spaces. The content of section 4 is devoted to the establishment of some advanced properties of the classical Sobolev space extended to the newly constructed spaces in terms of density results, extension theorems, and some continuous embedding theorems.

\section{Definitions and Preliminaries}

This section comprises of the existing notions of fractional Sobolev spaces and spaces of tempered ultradistributions, and see \cite{10} and \cite{1} for more detailed accounts on the respective subject areas. We present some existing definitions of fractional Sobolev space, space of ultradistributions of slow growth, and some properties of the theories. Throughout this work $\Omega$ denotes a finite or infinite subsets of $\mathbb{C}$ and $\mathbb{N}$ denotes the usual natural numbers set. In the next two subheadings, we present some elementary definitions and properties of classical Sobolev with fractional orders $\mathcal{W}^{\beta,p}(\Omega)$, where $\beta$ is an arbitrary real number, $0<\beta<1$ and $p\in (1,\infty)$, with $p\neq 2$.

\subsection{Structural Notion of Fractional Sobolev Space}
We introduce some versions of the existing definitions of the subject area. For more details on the definitions and underlying properties, see \cite{4,10}.
\begin{Def}
\label{def1}
Let $p\in [1,+\infty)$ and $\beta\in (0,1)$. The classical Sobolev space with fractional order $\mathcal{W}^{\beta,p}(\Omega)$ is defined as the space
\begin{eqnarray}
\label{eqn2.1}
\mathcal{W}^{\beta,p}(\Omega):=\Big\{ \mu\in L^{p}(\Omega):[\mu]_{\mathcal{W}^{\beta,p}(\Omega)}=\Bigg(\int_{\Omega}\int_{\Omega}\frac{|\mu(\xi)-\mu(\eta)|^{p}}{|\xi-\eta|^{n+p\beta}}d\xi d\eta\Bigg)^{\frac{1}{p}}<+\infty,\quad \xi\neq \eta\Big\}
\end{eqnarray}
The space (\ref{eqn2.1}) is endowed with the following natural norm
\[\Vert\mu\Vert_{\mathcal{W}^{\beta,p}(\Omega)}:=\Vert\mu\Vert_{L^{p}(\Omega)}+[\mu]_{\mathcal{W}^{\beta,p}(\Omega)}\]
\end{Def}

\begin{Rem}
\label{rem1}
The space $\mathcal{W}^{\beta,p}(\Omega)$ defined in (\ref{eqn2.1}) does not necessarily have the distributional nature in the sense of the seminorm unlike the classical Sobolev space $\mathcal{W}^{1,p}(\Omega)$, $[\cdot]_{\mathcal{W}^{\beta,p}(\Omega)}$.
\end{Rem}

\begin{Def}
\label{def2.2}
A function $\mu\in L^{p}(\Omega)$ belongs to the family of the fractional Sobolev space $\mathcal{W}^{\beta,p}(\Omega)$ if the following condition is verified:
\[\Vert\mu\Vert_{\mathcal{W}^{\beta,p}(\Omega)}:=\Vert\mu\Vert_{L^{p}(\Omega)}+|\mu|_{\mathcal{W}^{\beta,p}(\Omega)}\]
where
\begin{eqnarray}
\label{eqn2.2}
|\mu|_{\mathcal{W}^{\beta,p}(\Omega)}:=\Bigg[ \int_{\Omega}\int_{\Omega}\frac{|\mu(\xi)-\mu(\eta)|^{p}}{|\xi-\eta|^{n+p\beta}}d\xi d\eta\Bigg]^{\frac{1}{p}}\quad\text{if}\quad 1\leq p<\infty.
\end{eqnarray}
We also have
\[|\mu|_{\mathcal{W}^{\beta,p}(\Omega)}:=\operatorname*{ess\,sup}_{\xi,\eta\in\Omega; \xi\neq \eta}\frac{|\mu(\xi)-\mu(\eta)|}{|\xi-\eta|^{\beta}}\quad\text{for}\quad p=\infty\]
\end{Def}
We see that for $p=2$, then $H^{\beta}(\Omega):=\mathcal{W}^{\beta,2}(\Omega)$ for a Hilbert space $H^{\beta}$.

\begin{Def}
\label{def2.3}
Let $\beta>0$ and $1\leq p\leq \infty$. Let $l:=[\beta]$, $-\beta-l\leq \beta-l$, and $\alpha:=\beta-l.$ The definition of $\mathcal{W}^{\beta,p}(\Omega)$ is also given as
\begin{eqnarray}
\label{eqn2.3}
\mathcal{W}^{\beta,p}(\Omega):=\Big\{ \mu\in \mathcal{W}^{l,p}(\Omega):\frac{|D^{l}\mu(\xi)-D^{l}\mu(\eta)|}{|\xi-\eta|^{\frac{1}{p}+\alpha}}\in L^{\Omega\times\Omega}=\frac{|D^{l}\Big(\mu(\xi)-\mu(\eta)\Big)|}{|\xi-\eta|^{\frac{1}{p}+\alpha}}\in L^{p}(\Omega\times\Omega)\Big\}
\end{eqnarray}
naturally endowed with the norm
\begin{eqnarray}
\label{labeleqn2.3*}
\Vert\mu\Vert_{\mathcal{W}^{\beta,p}(\Omega)}:=\begin{cases}
\Big(\Vert\mu\Vert_{\mathcal{W}^{l,p}(\Omega)}+|D^{l}\mu|_{\mathcal{W}^{\beta,p}(\Omega)}\Big)^{\frac{1}{p}}\quad\text{if}\quad 1\leq p<\infty\\
\Vert\mu\Vert_{\mathcal{W}^{l,\infty}(\Omega)}+|D^{l}\mu|_{\mathcal{W}^{\beta,\infty}(\Omega)}\quad\text{if}\quad p=\infty
\end{cases}
\end{eqnarray}
where
\begin{eqnarray}
\label{labeleqn2.3**}
|\mu|_{\mathcal{W}^{\beta,p}(\Omega)}:=\begin{cases}
\Big(\int_{\Omega}\int_{\Omega}\frac{|\mu(\xi)-\mu(\eta)|^{p}}{|\xi-\eta|^{1+\beta p}}d\xi d\eta\Big)^{\frac{1}{p}}\quad\text{if}\quad 1\leq p<\infty\\
\displaystyle\operatorname*{ess\,sup}_{(\xi,\eta)\in\Omega\times\Omega}\frac{|\mu(\xi)-\mu(\eta)|}{|\xi-\eta|^{\beta}}\quad\text{if}\quad p=\infty
\end{cases}
\end{eqnarray}
and $D^l$ denotes the usual classical fractional derivatives.
\end{Def}
We also present another version of fractional Sobolev space defined via tempered distribution as follows

\begin{Def}
\label{def2.4}
Define the fractional Sobolev space $\mathcal{W}^{\beta,p}(\mathbb{R})$, if $\beta>0$ and $1\leq p\leq \infty$, by
\[\mathcal{W}^{\beta,p}(\mathbb{R})=\Big\{\mu\in L^{\mathbb{R}}:|\mu|_{\mathcal{W}^{\beta,p}(\mathbb{R})}<\infty\Big\}\]
where
\begin{eqnarray}
\label{eqn2.4}
|\mu|_{\mathcal{W}^{\beta,p}(\mathbb{R})}=\int_{\mathbb{R}}(1+|x|^{\beta p})|\hat{\mu}(x)|^{p}dx
\end{eqnarray}
Similarly, if $p=2$, then we have $H^{p}(\mathbb{R}):=\mathcal{W}^{\beta,2}(\mathbb{R})$.
\end{Def}

\subsection{Notion of Ultradistributions of Slow Growth}
We present the definition of more general Schwartz spaces of distributions known as the tempered ultradistributions. For more detailed information on different versions of tempered ultradistributions and their properties, refer to \cite{1,8}.
\begin{Def}
\label{def2.5}
Let $D^{i}$ denote the ultradifferential operator with respect to the multi-index $i=(i_{1}.i_{2},\cdots\,i_{n})$. A function $\varphi$ is said to be rapidly decreasing or of a rapid descent or of a slow increase growth in $\mathbb{C}^n$ if $\varphi\in C^{\infty}(\mathbb{C}^{n})$ and if the following condition is verified:
\begin{eqnarray}
\label{eqn2.5}
\forall\quad i\in\mathbb{N}^{n},\quad\forall\quad \beta\in\mathbb{N},\quad (1+|\xi|^{\beta})D^{i}\varphi\in L^{\infty}(\mathbb{C}^{n})
\end{eqnarray}
Note that the set of functions that satisfy the condition (\ref{def2.5}) can be denoted by $\mathcal{U}(\mathbb{C}^{n})$ and they are called ultradifferentiable functions satisfying growth condition. Also recall that $\mathcal{U}(\mathbb{C})|_{\mathbb{R}}=\mathcal{S}(\mathbb{R})$. \ \\
\noindent The condition (\ref{def2.5}) can also be formulated in two forms:
\[\forall\quad (i,\beta),\quad (1+|\xi|^{\beta})|D^{i}\varphi|\in L^{1}(\mathbb{C}^{n})\]
or
\[\forall\quad (i,\beta),\quad \lim_{|\xi|\to\infty}(1+|\xi|^{\beta})|D^{i}\varphi|=0\]
\end{Def}

\begin{Rem}
The topological vector space $\mathcal{U}(\mathbb{C}^{n})$ is naturally endowed with a seminorm (countable number of seminorms (norms)) $\eta_{\beta,i}$ such that
\[\eta_{\beta,i}(\varphi)=\sup\Big(|\text{Re}(\xi)|^{m}|D^{k}\varphi(\xi)|\Big)<\infty\]
We let $\mathcal{U}'(\mathbb{C}^{n})$ denote the topological dual of $\mathcal{U}(\mathbb{C}^{n})$ which is a locally convex topological vector space and is a subspace of $\mathcal{D}'$, the space of distributions and a super space of $\mathcal{S}'(\mathbb{R}^{n})$, the space of tempered distributions, with an inclusion.
\end{Rem}
In what follows, we present some different versions of definitions of space of tempered ultradistributions $\mathcal{U}'(\mathbb{C}^{n})$
\begin{Def}
\label{def2.6}
The space $\mathcal{U}\subset\mathcal{S}$ is the space of functions in $\mathcal{S}$ that can also be extended into the complex plane as entire functions with rapid descend on strips. The space $\mathcal{U}(\mathbb{C})$ consists of all entire functions $\varphi$ such that the following condition is verified:
\begin{eqnarray}
\label{eqn2.6}
\Vert \varphi\Vert_{p}=\sup_{\substack{|\text{Im}(\xi)|<p\\ \xi\in \mathbb{C}}}\Big\{ (1+|\xi|^{p})|\varphi(\xi)|\Big\}<\infty\quad\forall\quad p\in \mathbb{N}
\end{eqnarray}
In particular, the topology of the space $\mathcal{U}(\mathbb{C})$ is generated by the norm defined in (\ref{def2.6}), and is Fr$\breve{\text{e}}$chet for which $\varphi(\xi)\in\mathcal{U}$, $\varphi(\xi)|_{\mathbb{R}}\in \mathcal{S}(\mathbb{R})$.
\end{Def}

\begin{Def}
\label{def2.7}
The topological dual space of the space $\mathcal{U}$ is called the space of tempered ultradistributions satisfying the condition (\ref{eqn2.6}).\ \\
An ultradistribution of slow growth is said to vanish in an open set $K\in\mathbb{C}$ if $\varphi(x+iy)-\varphi(x-iy)\longrightarrow 0$ for $x\in K$ whenever $y\to 0$. We also see that a tempered ultradistribution $\mu$ in $\mathbb{R}$ is compact supported if there exists a disk $D$ such that any $\varphi$ in $[\varphi(\xi)]\in\mathcal{U}'$ has an analytic extension to $\Big(\frac{\mathbb{C}}{D}\Big)^{n}$ being of at most polynomial growth there. Equivalently, an ultradistribution of slow growth $\mu$ has a compact support if there exists a unique representative function $\varphi'(\xi)$ vanishing at $\infty$.
\end{Def}

\begin{Rem}
The motivation for the development of tempered ultradistributions in the fractional Sobolev spaces emanated from the possibility to deal with the notions of smoothness and non-smoothness in functions in exchange of the differential operators with ultra-differential operators and lay emphasis on the control of the growth of the functions in $\widetilde{\mathcal{W}}^{\beta,p}$.
\end{Rem}

\subsection{Some Properties of Fractional Sobolev Spaces}
We refer to some known properties and characterizations exhibited by the classical Sobolev spaces with fractional orders as presented in \cite{2,10,6,11} . Some existing key results such as the density theorems, extension results, and embedding theorems are presented in the aforementioned references as also analogous to the integer order classical Sobolev spaces.\ \\
\noindent The presentation of the density properties of smooth and smoothly supported functions in the fractional Sobolev spaces is taken care of in the first instance. The relationship between the sets $C^{\infty}_{0}(\Omega)$ and $\mathcal{W}^{\beta,p}(\Omega)$ can be investigated since $C_{0}^{\infty}$ is defined as
\[C_{0}^{\infty}(\Omega):=\Big\{\mu:\mathbb{R}^{n}\longrightarrow\mathbb{R}^{n},\quad\text{supp}(\mu)\subseteq \Omega\quad\text{is compact}\Big\}\]
where $\text{supp}(\mu)$ has the usual definition or meaning.\ \\
The theory of convolution plays a vital role in the notion of the fractional Sobolev space such that for any sufficiently small $\varepsilon$, convolution does not change with respect to the norm (\ref{eqn2.3}) as presented in the following cited result.

\begin{Lem}(\cite{2})
\label{lem2.15}
Let $\displaystyle\mu\in \mathcal{W}^{\beta,p}(\mathbb{R}^{n})$. Then $\displaystyle \Vert\mu_{\varepsilon}-\mu\Vert_{\mathcal{W}^{\beta,p}(\Omega)}\longrightarrow 0$ as $\varepsilon$, where $\mu_{\varepsilon}(x)=(\mu \ast\xi_{\varepsilon})(y),\quad y\in\mathbb{R}^n$.
\end{Lem}

The following result is termed the density theorem centered on continuous boundary.
\begin{Th}(\cite{2})
\label{theo2.18}
Let $\Omega$ be an open subset of $\mathbb{R}^{n}$ with its boundary. Then, for any $\mu\in \mathcal{W}^{\beta,p}_{0}(\mathbb{R}^{n})$ there exists a sequence $\rho_{\varepsilon}\in C_{0}^{\infty}(\Omega)$ such that $\displaystyle\Vert\rho_{\varepsilon}-\mu\Vert\longrightarrow 0$ as $\varepsilon\to 0$. In other words, $C_{0}^{\infty}(\Omega)$ is dense subspace of $\mathcal{W}^{\beta,p}_{0}(\Omega)$.
\end{Th}

For more detailed information on some extension theorems, refer to \cite{11}, and for more details on some continuous embedding results, see \cite{4,6,11}. The extension results are hinged on both the interior and exterior results.

\section{On Fractional Sobolev Spaces of Tempered Ultradistributions $\widetilde{\mathcal{W}}_{\mathcal{U}}^{\beta,p}(\mathbb{C})$}
In this section we introduce the extended family of fractional Sobolev space via the use of ultradistributions of slow growth. We shall give some versions of the new family with their respective norms. Results are given based on the extension from the integer order cases to the fractional order cases within framework of tempered ultradistributions. For some details on the classical Sobolev spaces using tempered ultradistributions, refer to \cite{3}. In the sequel, let $\Omega$ be a domain in the complex plane $\mathbb{C}$ or the complex domain.

\begin{Def}
\label{def3.1}
Given $\beta>0$ and $1\leq p\leq \infty$. We define the fractional Sobolev space of tempered ultradistributions $\hat{\mu}:\mathcal{U}\longrightarrow \mathbb{C}$ with respective to ultra-differentiable function $\mu$ as
\begin{eqnarray}
\label{eqn3.1}
\widetilde{\mathcal{W}}_{\mathcal{U}}^{\beta,p}(\mathbb{C}^{n}):=\Big\{ \mu\in \mathcal{U}(\mathbb{C}^{n})\quad\text{and}\quad \mu\in L^{p}(\Omega): [\mu]_{\widetilde{\mathcal{W}}_{\mathcal{U}}^{\beta,p}(\mathbb{C}^{n})}<\infty\Big\}
\end{eqnarray}
where
\[[\mu]_{\widetilde{\mathcal{W}}_{\mathcal{U}}^{\beta,p}(\mathbb{C}^{n})}=\int_{\mathbb{C}^{n}}\Big(1+|\xi|^{n+p\beta}\Big)|\hat{\mu}(\xi)|^{p}d\xi,\quad\text{for}\quad\xi\in\mathbb{C}^n\]
\end{Def}
We also present another version of Definition \ref{def3.1} for a complex open subset $\Omega$ of $\mathbb{C}^n$.

\begin{Def}
\label{def3.2}
Let $p\in [1,\infty]$ and $\beta\in (0,1)$. The new fractional Sobolev space can also be described as
\begin{eqnarray}
\label{eqn3.2}
\widetilde{\mathcal{W}}_{\mathcal{U}}^{\beta,p}(\Omega):=\Big\{ \mu\in \mathcal{U}(\mathbb{C}^{n}):\Big(\int_{\Omega}\int_{\Omega}\frac{(1+|\xi|^{n+p\beta})|\mu(\xi)-\mu(\eta)|^{p}}{|\xi-\eta|^{n+p\beta}}d\xi d\eta\Big)^{\frac{1}{p}}   <\infty,\quad \xi\neq \eta\Big\}
\end{eqnarray}
The space (\ref{eqn3.2}) defined in Definition \ref{def3.1} is endowed with the natural norm
\[\Vert\mu\Vert_{\widetilde{\mathcal{W}}_{\mathcal{U}}^{\beta,p}(\Omega)}=\Vert\mu\Vert_{L^{p}(\Omega)}+[u]_{\widetilde{\mathcal{W}}_{\mathcal{U}}^{\beta,p}(\Omega)}\]
For $p=\infty$, we have
\[\Vert\mu\Vert_{\widetilde{\mathcal{W}}_{\mathcal{U}}^{\beta,\infty}(\Omega)}=\Vert\mu\Vert_{L^{\infty}(\Omega)}+[u]_{\widetilde{\mathcal{W}}_{\mathcal{U}}^{\beta,\infty}(\Omega)}\]
where
\[[u]_{\widetilde{\mathcal{W}}_{\mathcal{U}}^{\beta,\infty}(\Omega)}=\sup_{(\xi,\eta)\in\Omega\times\Omega}\frac{(1+|\xi|^{n+\beta})|\mu(\xi)-\mu(\eta)|}{|\xi-\eta|^{n+\beta}}\]
\end{Def}
It is natural to define the classical Sobolev space with fractional order of tempered ultradistributions with reference to weak fractional ultra-derivatives or fractional derivatives and condition \ref{def2.6} in the same manner as done in the integer order case. We now introduce our fractional Sobolev spaces of tempered ultradistributions using the weak fractional derivative analogy as follows.

\begin{Def}
\label{def3.3}
For $\beta>0$, $1\leq p\leq \infty$, $|\beta|\leq k$, the fractional Sobolev space $\widetilde{\mathcal{W}}_{\mathcal{U}}^{\beta,p}(\Omega)$ is defined by
\begin{eqnarray}
\label{eqn3.3}
\widetilde{\mathcal{W}}_{\mathcal{U}}^{\beta,p}(\Omega):=\Big\{\mu\in \widetilde{\mathcal{W}}_{\mathcal{U}}^{k,p}(\Omega)\quad\text{and}\quad\mu\in \mathcal{U}(\Omega): D^{\beta}\Big((1+|\xi|^{n+p\beta})\hat{\mu}\Big)\in L^{p}(\Omega)\Big\}
\end{eqnarray}
for $\Omega\subset\mathbb{C}^n$ endowed with the respective norms:\ \\
For $1\leq p<\infty$, we have
\[\Vert\mu\Vert_{\widetilde{\mathcal{W}}_{\mathcal{U}}^{\beta,p}(\Omega)}=\Big(\Vert\mu\Vert_{\widetilde{\mathcal{W}}_{\mathcal{U}}^{k,p}(\Omega)}^{p}+\Vert D^{\beta}(1+|\xi|^{n+p\beta}\hat{\mu})\Vert_{L^{p}(\Omega)}\Big)^{\frac{1}{p}}\]
For $p=\infty$, we have
\[\Vert\mu\Vert_{\widetilde{\mathcal{W}}_{\mathcal{U}}^{\beta,\infty}(\Omega)}=\Vert\mu\Vert_{\widetilde{\mathcal{W}}_{\mathcal{U}}^{k,\infty}(\Omega)}+\Vert D^{\beta}(1+|\xi|^{n+\beta}\hat{\mu})\Vert_{L^{\infty}(\Omega)}\]
\end{Def}

It is easy to see that the ultra-derivatives of a tempered ultradistribution, and the product of a tempered ultradistributions and a slowly increasing ultra-differentiable function are also tempered ultradistributions. We present the first property involving the completeness of the fractional Sobolev spaces of ultradistributions of slow growth.

\begin{Prop}
\label{prop3.4}
Let $0<\beta<1$ and $1\leq p\leq \infty$, and let $\mu\in\mathcal{U}'(\Omega)$ satisfying condition \ref{eqn2.6}. Then the space $\widetilde{\mathcal{W}}_{\mathcal{U}}^{\beta,p}(\Omega)$ endowed with the norm
\[\Vert\mu\Vert_{\widetilde{\mathcal{W}}_{\mathcal{U}}^{\beta,p}(\Omega)}=\Big(\Vert\mu\Vert_{L^p}^{p}+[\mu]_{\widetilde{\mathcal{W}}_{\mathcal{U}}^{\beta,p}(\Omega)}\Big)^{\frac{1}{p}}\]
where
\[[\mu]_{\widetilde{\mathcal{W}}_{\mathcal{U}}^{\beta,p}}=\Big(\int_{\Omega}\int_{\Omega}\frac{(1+|\xi|^{n+p\beta})|\mu(\xi)-\mu(\eta)|^{p}}{|\xi-\eta|^{n+p\beta}}d\xi d\eta\Big)^{\frac{1}{p}}   <\infty,\quad \xi\neq \eta\]
is a Banach space.
\end{Prop}

\begin{proof}
We show that $\widetilde{\mathcal{W}}_{\mathcal{U}}^{\beta,p}(\Omega)$ is Banach. To do this, let $\lbrace\mu_{n}\rbrace$ be a Cauchy sequence in $\widetilde{\mathcal{W}}_{\mathcal{U}}^{\beta,p}(\Omega)$ for the norm $\Vert\mu\Vert_{\widetilde{\mathcal{W}}_{\mathcal{U}}^{\beta,p}(\Omega)}$.\ \\
Since $\lbrace\mu_{n}\rbrace$ is Cauchy in $L^p$, then $\mu_{n}\longrightarrow \mu$ in $L^{p}$ for $1\leq p\leq \infty$. Hence the sequence. Hence the sequence
\begin{eqnarray}
\label{eqn3.4}
\nu_{n}(\xi,\eta)=\frac{(1+|\xi|^{n+p\beta})\Big(\mu(\xi)-\mu(\eta)\Big)}{|\xi-\eta|^{\frac{n}{\beta}+p}}
\end{eqnarray}
is Cauchy in $L^p$ which converge to an element
\begin{eqnarray}
\label{eqn3.5}
\nu(\xi,\eta)=\frac{(1+|\xi|^{p\beta})\Big(\mu(\xi)-\mu(\eta)\Big)}{|\xi-\eta|^{p}}
\end{eqnarray}
in $L^p$. Now it suffice to create a subsequence $\lbrace\mu_{t_{n}}\rbrace$ of $\lbrace\mu_{n}\rbrace$ that converges almost everywhere to $\mu$. Recall that the sub-sequence to the sequence defined in (\ref{eqn3.4}) converges almost every pair $(\xi,\eta)$ to (\ref{eqn3.5}).\ \\
With the application of the condition (\ref{eqn2.6}) and Fatou's lemma, we have
\begin{align*}
\int_{\Omega}\int_{\Omega}\frac{(1+|\xi|^{n+p\beta})|\mu(\xi)-\mu(\eta)|^{p}}{|\xi-\eta|^{n+p\beta}}d\xi d\eta &\leq \sup_{\substack{|\text{Im}(\xi)|<p\beta\\ \xi,\eta\in\mathbb{C}}} \int_{\Omega}\int_{\Omega}\frac{(1+|\xi|^{n+p\beta})|\mu(\xi)-\mu(\eta)|^{p}}{|\xi-\eta|^{n+p\beta}}d\xi d\eta\\
     &=\sup_{\substack{|\text{Im}(\xi_{1})|<p\beta\\ \xi_{1}\in\mathbb{C}}} \int_{\Omega}\frac{(1+|\xi_{1}|^{n+p\beta})|\varphi(\xi_{1})|^{p}}{|\xi_{1}|^{n+p\beta}}d\xi_{1}<\infty
\end{align*}
Hence $\displaystyle \mu\in \widetilde{\mathcal{W}}_{\mathcal{U}}^{\beta,p}(\Omega)$. Therefore $\mu\to\mu$ in $\widetilde{\mathcal{W}}_{\mathcal{U}}^{\beta,p}(\Omega)$ by taking $m\longrightarrow\infty$ in $\displaystyle\Vert\nu_{n}-\nu_{m}\Vert_{L^{p}(\Omega\times\Omega)}$.
\end{proof}

The next result deduce the condition for which tempered ultradistribution can be included in $\widetilde{\mathcal{W}}_{\mathcal{U}}^{\beta,p}(\Omega)$.

\begin{Prop}
\label{prop3.5}
The $\widetilde{\mathcal{W}}_{\mathcal{U}}^{\beta,p}(\Omega)$ is said to of the local type if for every $\mu\in \widetilde{\mathcal{W}}_{\mathcal{U}}^{\beta,p}(\Omega)$ and for every $\varphi\in\mathcal{U}(\Omega)$, then $\mu\varphi\in \widetilde{\mathcal{W}}_{\mathcal{U}}^{\beta,p}(\Omega)$ such that $\displaystyle \sup_{|\text{Im}(\xi)|<p\beta}\Big\{(1+|\xi|^{p\beta})|\mu\varphi(\xi)|^{p}\Big\}<\infty$.
\end{Prop}

\begin{proof}
Let $\mu\in \widetilde{\mathcal{W}}_{\mathcal{U}}^{\beta,p}(\Omega)$ and $\varphi\in \mathcal{U}(\Omega)$. Clearly we have that $\mu\varphi\in L^{p}$ since $\mathcal{U}\subset L^{p}$. We show that
\[\sup_{\substack{|\text{Im}(\xi)|<p\beta\\ \xi,\eta\in\mathbb{C}}} \int_{\Omega}\int_{\Omega}\frac{(1+|\xi|^{n+p\beta})|\mu(\xi)-\mu(\eta)|^{p}}{|\xi-\eta|^{n+p\beta}}d\xi d\eta<\infty\]
Since $\displaystyle \mathcal{U}\ni \varphi\mapsto (1+|\xi|^{n+p\beta})|\hat{\varphi}(\xi)|^{p}$ defines a tempered ultradistribution in $\widetilde{\mathcal{W}}_{\mathcal{U}}^{\beta,p}(\Omega)$ for $|\beta|\leq k$ and for which $\beta\in (0,1)$, and note that $\displaystyle (1+|\xi|^{n+p\beta})\varphi(\xi)[\mu(\xi)-\mu(\eta)]$ determines a convergent integral since $(1+|\xi|^{n+p\beta})\varphi(\xi)$ is bounded or finite.\ \\
Then it follows that
\begin{align*}
\int_{\Omega}\int_{\Omega}\Big|\frac{(1+|\xi|^{n+p\beta})\Big(\varphi(\xi)-\varphi(\eta)\Big)\mu(\eta)}{|\xi-\eta|^{n+p\beta}}\Big|^{p}d\xi d\eta & \leq\Vert\varphi''\Vert_{\infty}^{p} \sup_{\substack{|\text{Im}(\xi)|<p\beta\\ \xi,\eta\in\mathbb{C}}} \int_{\Omega}(1+|\xi|^{n+p\beta})|\mu(\eta)|^{p}\\
&\Big[\int_{|\text{Im}(\xi)|<p\beta}|\xi-\eta|^{(-n-p\beta)+p}d\xi\Big]d\eta\\
& \leq Mt^{p(1-\beta)}\Vert\mu\Vert_{L^{\Omega}}
\end{align*}
where $\varphi''=\varphi(\xi)-\varphi(\eta)$ and $\displaystyle \sup_{\xi,\eta\in\Omega}\Vert\xi-\eta\Vert$.
\end{proof}

\section{Some Results on Density, Extension, and Embeddings in $\widetilde{\mathcal{W}}_{\mathcal{U}}^{\beta,p}(\mathbb{C})$}

We begin with some results on density, and we consider the extension of ultradistributions of slow growth in $\widetilde{\mathcal{W}}_{\mathcal{U}}^{\beta,p}(\mathbb{C})$ with compact support. The relationship between the properties of density, extension and embedding will be presented.

\begin{Lem}
\label{lem3.10}
Let $(c,d)\in\Omega\subset\mathbb{C}^{n}$, $\beta>0$, $1\leq p<\infty$. Suppose $\mu\in \widetilde{\mathcal{W}}_{\mathcal{U}}^{\beta,p}(\Omega)$ and $\text{supp}(\mu)=K\subset\Omega$ then the extension $\tilde{\mu}$ belongs to $\widetilde{\mathcal{W}}_{\mathcal{U}}^{\beta,p}(\Omega)$ and there is a constant $M=M(\beta,p,K)>0$ such that the following condition is verified:
\begin{eqnarray}
\label{eqn3.10}
\Vert\tilde{\mu}\Vert_{\widetilde{\mathcal{W}}_{\mathcal{U}}^{\beta,p}(\Omega)}\leq M\Vert\mu\Vert_{\widetilde{\mathcal{W}}_{\mathcal{U}}^{\beta,p}(\Omega)}
\end{eqnarray}
\end{Lem}

\begin{proof}
We need prove that the inequality (\ref{eqn3.10}) is valid for the extension $\tilde{\mu}$. To do this, let $\lbrace\tilde{\mu}_{i}\rbrace$ be a sequence of tempered ultradistributions in $\widetilde{\mathcal{W}}_{\mathcal{U}}^{\beta,p}(\Omega)$ such that $\displaystyle\lbrace\tilde{\mu}_{i}\rbrace\mapsto (1+|\xi|^{n+p\beta})\mu_{i}\longrightarrow (1+|\xi|^{n+\beta p})\mu$ in $C_{0}^{\infty}(\Omega)$ as $i\to\infty$. We can define the sequence
\[\tilde{\mu}_{i}(\xi)=\begin{cases}
(1+|\xi|^{n+p\beta})\mu_{i}(\xi), & xi\in\Omega\\
0, & \xi\notin\mathbb{C}\setminus\Omega
\end{cases}
\]
Therefore, let $\displaystyle \text{supp}(\Omega)\subset K\subset \Omega$ so that $\displaystyle\sup_{|\text{Im}(\xi)|<p\beta}\Big\{(1+|\xi|^{n+p\beta})\mu_{i}(\xi)\Big\}<\infty$ and we take the ultra-derivative as follows
\begin{align*}
\Vert D^{\beta}\tilde{\mu}_{i}\Vert_{L^{p}(\mathbb{C})} &= \Vert D^{\beta}(1+|\xi|^{n+p\beta})\tilde{\mu}_{i}\Vert_{L^{p}(c,d)}+\Vert L(1+|\xi|^{n+p\beta})\tilde{\mu}_{i}\Vert_{L^{p}(d,\infty)}\\
   &=\Vert D^{\beta}(1+|\xi|^{n+p\beta})\tilde{\mu}_{i}\Vert_{L^{p}(c,d)}+\int_{d}^{\infty}\Big|\int_{e}^{f}\frac{(1+|\xi|^{n+p\beta})\mu_{i}(\xi)}{(\xi-\eta)^{n+p\beta}}d\xi\Big|d\eta\\
   &\leq \Vert D^{\beta}(1+|\xi|^{n+p\beta})\tilde{\mu}_{i}\Vert_{L^{p}(c,d)}+\Vert \tilde{\mu}_{i}\Vert_{L^{p}(c,d)}\cdot \frac{(1+|d|^{n+p\beta+1})(d-c)^{\frac{p}{p'}}}{(f-e)^{p(n+\beta)-1}}
\end{align*}
Then there is a constant $M(\beta,p,K)>0$ such that
\[\Vert D^{\beta}\tilde{\mu}_{i}\Vert_{L^{p}(\mathbb{C})}\leq M\Vert\mu_{i}\Vert_{\widetilde{\mathcal{W}}_{\mathcal{U}}^{\beta,p}(\Omega)}\]
We see that $\displaystyle \Vert\mu_{i}\Vert_{\widetilde{\mathcal{W}}_{\mathcal{U}}^{\beta,p}(\Omega)}\longrightarrow\Vert\mu\Vert_{\widetilde{\mathcal{W}}_{\mathcal{U}}^{\beta,p}(\Omega)}$ and $\displaystyle \Vert\tilde{\mu}_{i}\Vert_{\widetilde{\mathcal{W}}_{\mathcal{U}}^{\beta,p}(\Omega)}<\varepsilon$ for all $i>N(\varepsilon)$.\ \\
suppose $\varepsilon>0$ is given and for sufficiently large $n$, we have
\[\Vert\tilde{\mu}_{m}-\tilde{\mu}_{n}\Vert_{\widetilde{\mathcal{W}}_{\mathcal{U}}^{\beta,p}(\Omega)}\leq \Vert\mu_{m}-\mu_{n}\Vert_{\widetilde{\mathcal{W}}_{\mathcal{U}}^{\beta,p}(\Omega)}<\varepsilon \]
This shows that $\lbrace\tilde{\mu}\rbrace$ is Cauchy. Then for $m\to\infty$, we see that
\[\Vert\tilde{\mu}_{n}-\mu\Vert_{\widetilde{\mathcal{W}}_{\mathcal{U}}^{\beta,p}(\Omega)}<\varepsilon\]
such that $\mu_{i}\to\upsilon$ in $\widetilde{\mathcal{W}}_{\mathcal{U}}^{\beta,p}(\mathbb{C})$. That is, for sufficiently large $i$, we have
\begin{align}
\label{eqn**}
\Vert\tilde{\mu}-\upsilon\Vert_{L^{p}(\mathbb{C})} &=\Vert\tilde{\mu}-\tilde{\mu}_{i}+\tilde{\mu}_{i}-\upsilon\Vert_{L^{p}(\mathbb{C})}\nonumber\\
   &\leq \Vert\tilde{\mu}-\tilde{\mu}_{i}\Vert_{L^{p}(\mathbb{C})}+\Vert\tilde{\mu}_{i}-\upsilon\Vert_{L^{p}(\mathbb{C})}\nonumber\\
   &<\varepsilon
\end{align}
Thus (\ref{eqn**}) occurs when $\displaystyle \Vert\tilde{\mu}-\upsilon\Vert_{L^{p}(\mathbb{C})}=0$. This shows that $\tilde{\mu}=\upsilon$ almost everywhere and this shows that the extension $\tilde{\mu}$ satisfies the desired conditions on compactly supported function.
\end{proof}

The next result called the density result establishes the extension from the domain $\Omega$ to the whole of $\mathbb{C}^n$.

\begin{Prop}
\label{prop3.6}
The space $\mathcal{D}$ is dense in $\widetilde{\mathcal{W}}_{\mathcal{U}}^{\beta,p}(\mathbb{C}^{n})$. In particular, $\mathcal{U}(\mathbb{C}^{n})$ is dense in $\widetilde{\mathcal{W}}_{\mathcal{U}}^{\beta,p}(\mathbb{C}^{n})$.
\end{Prop}

\begin{proof}
We define an ultra-differentiable function $\varphi$ in $\widetilde{\mathcal{W}}_{\mathcal{U}}^{\beta,p}(\mathbb{C}^{n})$ such that the sequence $\varphi_{n}$ in $\mathcal{U}(\mathbb{C}^{n})$ converges to $\varphi$. To see this, the density of $\mathcal{D}$ in $L^p$ implies that there is existence of sequence $\lbrace\varphi_{n}\rbrace\in\mathcal{D}(\mathbb{C})$ such that
\[\lim_{n\to\infty}\Vert\varphi_{n}-(1+|\xi|^{n+p\beta})\hat{\mu}\Vert_{p}=0\]
Thus the sequence of functions $\nu_{n}=(1+|\xi|^{n+p\beta})\varphi_{n}$ belongs to $\mathcal{D}$ and automatically converges to $\hat{\mu}$ in $L^p$. By the continuity of the inverse Fourier transform, we conclude that $\breve{\nu}_{n}$ is in $\mathcal{U}(\mathbb{C}^{n})$ and hence converges to $\mu$ in $\widetilde{\mathcal{W}}_{\mathcal{U}}^{\beta,p}(\mathbb{C}^{n})$ provided $\displaystyle \sup_{\substack{ \xi\in\mathbb{C}\\  |\text{Im}(\xi)|<p\beta}}(1+|\xi|^{n+p\beta})|\hat{\mu}(\xi)|<\infty$.
\end{proof}

\begin{Rem}
From Lemma \ref{lem3.10}, there is a number associated with the extension $\tilde{\mu}$ given by
\[M=\sup_{\xi\in\Omega}\Bigg(\frac{\Vert\mu\Vert_{\widetilde{\mathcal{W}}_{\mathcal{U}}^{\beta,p}(\Omega)}}{\Vert\mu\Vert_{\widetilde{\mathcal{W}}_{\mathcal{U}}^{\beta,p}(\Omega)}}\Bigg)\]
It follows that $\tilde{\mu}$ is bounded by a number $M$. Similarly fractional Sobolev tempered ultradistribution $\mu$ admit a bounded inverse extension $\mu\mapsto\tilde{\mu}$ as presented in the following result.
\end{Rem}

\begin{Th}
\label{theo3.12}
Let $\Omega\subset\mathbb{C}$, $0<\beta<1$, and $p\in [1,\infty)$. Then the extension $\mu\mapsto\tilde{\mu}$ admits a bounded inverse if there is a constant $M=M(\beta,p,K)>0$ such that
\[M\Vert\mu\Vert_{\widetilde{\mathcal{W}}_{\mathcal{U}}^{\beta,p}(\Omega)}\leq \Vert\tilde{\mu}\Vert_{\widetilde{\mathcal{W}}_{\mathcal{U}}^{\beta,p}(\Omega)}\]
where $K=\text{supp}(\mu)\subset\Omega$.
\end{Th}
\begin{proof}
The extended function $\tilde{\mu}$ having satisfied the desired properties, we show that $\mu$ admit a bounded inverse. From the continuity of $\mu$ and $\tilde{\mu}$ in $\widetilde{\mathcal{W}}_{\mathcal{U}}^{\beta,p}(\Omega)$, then $\tilde{\mu}^{-1}\eta=\mu$ and we have $\displaystyle \mu\mapsto (1+|\xi|^{n+p\beta})\mu$ and
\[\Vert\tilde{\mu}^{-1}\eta\Vert_{\widetilde{\mathcal{W}}_{\mathcal{U}}^{\beta,p}(\Omega)}=\Vert\mu\Vert_{\widetilde{\mathcal{W}}_{\mathcal{U}}^{\beta,p}(\Omega)}\leq \frac{1}{M}\Vert\tilde{\mu}\Vert_{\widetilde{\mathcal{W}}_{\mathcal{U}}^{\beta,p}(\Omega)}\]
Then
\[\Vert\tilde{\mu}^{-1}\Vert_{\widetilde{\mathcal{W}}_{\mathcal{U}}^{\beta,p}(\Omega)}\leq \frac{1}{M}\Vert\tilde{\mu}\Vert_{\widetilde{\mathcal{W}}_{\mathcal{U}}^{\beta,p}(\Omega)}\]
By choosing $\frac{1}{M}=C$, we have
\[\Vert\tilde{\mu}^{-1}\Vert_{\widetilde{\mathcal{W}}_{\mathcal{U}}^{\beta,p}(\Omega)}\leq C\Vert\tilde{\mu}\Vert_{\widetilde{\mathcal{W}}_{\mathcal{U}}^{\beta,p}(\Omega)}\]
Hence the result.
\end{proof}

\begin{Cor}
\label{cor3.13}
All extension operators from finite fractional Sobolev space $\widetilde{\mathcal{W}}_{\mathcal{U}}^{\beta,p}(\Omega)$ into any arbitrary fractional Sobolev space is bounded.
\end{Cor}

\begin{proof}
We show that a sequence of tempered ultradistributions in $\widetilde{\mathcal{W}}_{\mathcal{U}}^{\beta,p}(\Omega)$ converges. Let $\lbrace(1+|\xi|^{p\beta})\mu_{i}\rbrace$ be a sequence from a finite space $\widetilde{\mathcal{W}}_{\mathcal{U}}^{\beta,p}(\Omega)$. Since the space is finite, there exists a basis $\lbrace e_{i}\rbrace_{i=1}^{n}$ such that each member in $\widetilde{\mathcal{W}}_{\mathcal{U}}^{\beta,p}(\Omega)$ can be a linear combination of $e_i$ such that $\displaystyle\sup_{\xi\in\mathbb{C}}\Big\{(1+|\xi|^{p\beta})\mu_{i}\Big\}_{i=1}^{n}<\infty$ and $E:\mu\longrightarrow\tilde{\mu}$ an extension with
\[\mu_{i}(\xi)=\sum_{i=1}^{n}\beta_{i}^{(n)}e_{i}(1+|\xi_{i}|^{p\beta})\]
and
\[\mu(\xi)=\sum_{i=1}^{\infty}\beta_{i}e_{i}(1+|\xi_{i}|^{p\beta})\quad\text{for}\quad \beta_{i}\in (0,1)\]
Therefore
\begin{align*}
E(\mu_{i}(\xi))& =\tilde{\mu}_{i}(\xi)=E\Big(\sum_{i=1}^{n}\beta_{i}^{(n)}e_{i}(1+|\xi_{i}|^{p\beta})\Big)\\
   &=\sum_{i=1}^{n}\beta_{i}^{(n)}E\Big(e_{i}(1+|\xi_{i}|^{p\beta})\Big)\to \sum_{i=1}^{\infty}\beta_{i}E\Big(e_{i}(1+|\xi_{i}|^{p\beta})\Big)\\
   &=E\Big(\sum_{i=1}^{\infty}\beta_{i}e_{i}(1+|\xi_{i}|^{p\beta})\Big)\\
   &=E(\mu)
\end{align*}
This complete the proof.
\end{proof}

There is an essential need to have an interior and exterior extensions in $\widetilde{\mathcal{W}}_{\mathcal{U}}^{\beta,p}(\Omega)$. For any tempered ultradistribution $\mu\in \widetilde{\mathcal{W}}_{\mathcal{U}}^{\beta,p}(\Omega)$ and $\Omega'\subset\Omega$, there is a re-arrangement for $\mu$ in $\Omega\setminus\Omega'$ such that the slow growth rearranged function $\mu^{\ast}$ has a compact support in $\Omega$ and coincides with $\mu$ in $\Omega$.

\begin{Lem}
\label{lem3.14}
Suppose that $\Omega$ is bounded and $0<\beta<1$. For any $\Omega'\subset\Omega\subset\mathbb{C}$ there is a compact subset $K\subset\Omega$ and a constant $M(\beta,K)>0$ such that for every $\mu\in \widetilde{\mathcal{W}}_{\mathcal{U}}^{\beta,p}(\Omega)$ there is an extension $E:\widetilde{\mathcal{W}}_{\mathcal{U}}^{\beta,p}(\Omega)\longrightarrow\widetilde{\mathcal{W}}_{\mathcal{U}}^{\beta,p}(\mathbb{C})$ satisfying the following properties
\begin{itemize}
\item[(i)] $E\mu=\mu^{\ast}$ almost everywhere in $\Omega'$ and $\text{supp}(E\mu)\subseteq K$
\item[(ii)] $\displaystyle \Vert E\mu\Vert_{\widetilde{\mathcal{W}}_{\mathcal{U}}^{\beta,p}(\mathbb{C})}\leq M\Vert\mu\Vert_{\widetilde{\mathcal{W}}_{\mathcal{U}}^{\beta,p}(\Omega)}$.
\end{itemize}
\end{Lem}

\begin{proof}
The proof is shown as follows: For any $\mu\in \widetilde{\mathcal{W}}_{\mathcal{U}}^{\beta,p}(\Omega)$, there is a rearrangement $\mu^{\ast}\in \widetilde{\mathcal{W}}_{\mathcal{U}}^{\beta,p}(\Omega)$ of $\mu$. Then let $M(\beta,K)$ such that $E\mu=\mu^{\ast}$ with the required properties. Since $K\subset\Omega$, then clearly $\text{supp}(E\mu)\subseteq K\subset\Omega'\subset\Omega$. The remaining proof this result is similar to the proof given in Lemma \ref{lem3.10}. Hence the result.
\end{proof}

\begin{Rem}
The extension $E\mu$ in Lemma \ref{lem3.14} is termed the interior extension of $\mu$ from $\Omega$ to the whole of $\mathbb{C}$. In a similar view, there is also need to construct the exterior extension to establish the relationship or coincidence between the original tempered ultradistribution and the extended ultradistribution of slow growth in the entire complex domain $\Omega$.
\end{Rem}
The next result explains the restriction on the tempered ultradistribution to the extended function with reference to the slow growth condition with ultra-polynomials.

\begin{Th}
\label{theo3.15}
Let $\Omega$ be a bounded domain, $0<\beta<1$ and $p\in [1,\infty)$. Suppose that $|\text{Im}(\xi)|<p\beta$ for $p\beta>1$ and $\mu\in\mathbb{C}$ such that $p(1-\beta p)<\mu$. Then there exists a constant $M$ which depends on $\beta,p$ and $\Omega$ for every bounded domain $\Omega'\subset\Omega$ such that for any $\mu\in \widetilde{\mathcal{W}}_{\mathcal{U}}^{\beta,p}(\Omega)\cap L^{p}(\Omega)$, there is a $\mu^{\ast}\in \widetilde{\mathcal{W}}_{\mathcal{U}}^{\beta,p}(\Omega')$ such that the following properties hold:
\begin{enumerate}
\item[(i)]$\displaystyle \sup\Big((1+|\xi|^{n+\beta p}|\mu^{\ast}(\xi)|^{p})\Big)<\infty$ for $\xi\in\Omega'$;
\item[(ii)] $\mu^{\ast}=\mu$ almost everywhere in $\Omega$;
\item[(iii)] $\text{supp}(\mu^{\ast})\subset\Omega'$;
\item[(iv)] $\displaystyle \Vert \mu^{\ast}\Vert_{\widetilde{\mathcal{W}}_{\mathcal{U}}^{\beta,p}(\Omega')}\leq M\Vert\mu\Vert_{\widetilde{\mathcal{W}}_{\mathcal{U}}^{\beta,p}(\Omega)}$
\end{enumerate}
\end{Th}

\begin{proof}
Assume that $\beta p<1$ and $|\text{Im}(\xi)|<p\beta$. Let $\Omega'$ be a bounded domain in $\Omega$ and $\mu\in \widetilde{\mathcal{W}}_{\mathcal{U}}^{\beta,p}(\Omega)\cap L^{p}(\Omega)$. To this end, let $\displaystyle\lbrace (1+|\xi|^{\beta p}\mu_{i})\rbrace\subset C^{\infty}(\Omega)$ such that $(1+|\xi|^{\beta p}\mu_{i})\mapsto (1+|\xi|^{\beta p}\mu)$ in $\widetilde{\mathcal{W}}_{\mathcal{U}}^{\beta,p}(\Omega)\cap L^{p}(\Omega)$ as $i\to\infty$. Hence $\displaystyle\lbrace (1+|\xi|^{\beta p}\mu_{i})\rbrace$ and $\displaystyle\lbrace D^{\beta}(1+|\xi|^{\beta p}\mu_{i})\rbrace$ are both bounded sequences in $L^{p}(\Omega')$ for $\Omega'\subset \Omega$. Since $\displaystyle D^{\beta}(1+|\xi|^{\beta p}\mu_{i})\longmapsto D^{\beta}(1+|\xi|^{\beta p}\mu)$ in $L^{p}(\Omega)$. Let $M>0$ be a bound for both sequences.\ \\
It follows that for sufficiently large $m,n$ and given $\varepsilon>0$ we have
\[\Vert D^{\beta}(1+|\xi|^{\beta p})[\mu_{m}-\mu_{n}]\Vert_{L^{p}(\Omega')}\leq M\Big( \Vert \mu_{m}-\mu_{n}\Vert_{\widetilde{\mathcal{W}}_{\mathcal{U}}^{\beta,p}(\Omega)}+\Vert \mu_{m}-\mu_{n}\Vert_{L^{p}(\Omega)}\Big)<\varepsilon\]
Therefore, there exists $\upsilon\in L^{p}(\Omega')$ such that $D^{\beta}(1+|\xi|^{\beta p}\mu_{i})\longrightarrow (1+|\xi|^{\beta p})\upsilon$ in $L^{p}(\Omega')$ and we see that $D^{\beta}(1+|\xi|^{\beta p}\mu_{i})= (1+|\xi|^{\beta p})\upsilon$ from the fractional weak derivative point of view. Hence $D^{\beta}\mu\in L^{p}(\Omega')$.
\end{proof}

The following result concerns the existence of embedding with emphasis on the fractional order.

\begin{Prop}
\label{prop3.7}
The following embeddings are satisfied by $\widetilde{\mathcal{W}}_{\mathcal{U}}^{\beta,p}$
\begin{itemize}
\item[(i)] if $0<\beta'<\beta<1$, then $\displaystyle \widetilde{\mathcal{W}}_{\mathcal{U}}^{\beta,p}(\Omega)\hookrightarrow \widetilde{\mathcal{W}}_{\mathcal{U}}^{\beta',p}(\Omega)$,
\item[(ii)] if $\beta\in(0,1)$, then $\displaystyle \mathcal{W}_{\mathcal{U}}^{1,p}(\Omega)\hookrightarrow \widetilde{\mathcal{W}}_{\mathcal{U}}^{\beta,p}(\Omega)$
\end{itemize}
In other words, (ii) shows that the integer order Sobolev space is continuously embedded in the fractional Sobolev space.
\end{Prop}

\begin{proof}
We prove property (i). We need to prove that for $\beta'<\beta\in (0,1)$ then $\displaystyle \widetilde{\mathcal{W}}_{\mathcal{U}}^{\beta,p}(\Omega)\hookrightarrow \widetilde{\mathcal{W}}_{\mathcal{U}}^{\beta',p}(\Omega)$.\ \\
Let $\mu$ satisfy condition (\ref{eqn2.6}) and belong to $\widetilde{\mathcal{W}}_{\mathcal{U}}^{\beta',p}(\Omega)$ such that $\mu\mapsto (1+|\xi|^{n+p\beta})|\hat{\mu}(\xi)|$ in $\widetilde{\mathcal{W}}_{\mathcal{U}}^{\beta',p}(\Omega)$. \ \\
Hence
\begin{align*}
\Vert\mu\Vert_{\widetilde{\mathcal{W}}_{\mathcal{U}}^{\beta',p}(\Omega)}&=\Vert\mu\Vert_{L^{p}(\Omega)}+|\mu|_{\widetilde{\mathcal{W}}_{\mathcal{U}}^{\beta',p}(\Omega)}\\
&=\Vert\mu\Vert_{L^{p}(\Omega)}+\Big(\int_{\Omega}\int_{\Omega}\frac{(1+|\xi|^{n+p\beta'})|\mu(\xi)-\mu(\eta)|^{p}}{|\xi-\eta|^{n+p\beta'}}d\xi d\eta\Big)^{\frac{1}{p}}\\
&\leq \Vert\mu\Vert_{L^{p}(\Omega)}+\sup_{\substack{\xi,\eta\in\mathbb{C}\\ |\text{Im}(\xi)|<p\beta'}}\Big(\int_{\Omega}\int_{\Omega}\frac{(1+|\xi|^{n+p\beta'})|\mu(\xi)-\mu(\eta)|^{p}}{|\xi-\eta|^{n+p\beta'}}d\xi d\eta\Big)^{\frac{1}{p}}\\
&\leq \Vert\mu\Vert_{L^{p}(\Omega)}+\sup_{\substack{\xi,\eta\in\mathbb{C}\\ |\text{Im}(\xi)|<p\beta}}\Big(\int_{\Omega}\int_{\Omega}\frac{(1+|\xi|^{n+p\beta})|\mu(\xi)-\mu(\eta)|^{p}}{|\xi-\eta|^{n+p\beta}}d\xi d\eta\Big)^{\frac{1}{p}}\\
&\leq \Vert\mu\Vert_{L^{p}(\Omega)}+\Big(\int_{\Omega}\int_{\Omega}\frac{(1+|\xi|^{n+p\beta})|\mu(\xi)-\mu(\eta)|^{p}}{|\xi-\eta|^{n+p\beta}}d\xi d\eta\Big)^{\frac{1}{p}}\\
&=\Vert\mu\Vert_{L^{p}(\Omega)}+|\mu|_{\widetilde{\mathcal{W}}_{\mathcal{U}}^{\beta,p}(\Omega)}
\end{align*}
for $n+p\beta'<n+p\beta$ which implies that $p\beta<p\beta'$.\ \\
Hence
\[\Vert\mu\Vert_{\widetilde{\mathcal{W}}_{\mathcal{U}}^{\beta',p}(\Omega)}\leq\Vert\mu\Vert_{\widetilde{\mathcal{W}}_{\mathcal{U}}^{\beta,p}(\Omega)} \]
The proof of (ii) is similar to the proof of (i).
\end{proof}

\begin{Rem}
From Proposition \ref{prop3.7}, we obtain the following density result.
\end{Rem}

\begin{Prop}
\label{prop3.8}
Let $\beta\in (0,1)$ and $p>0$. Suppose that an open subset $\Omega\subset\mathbb{C}$ admit an extension, then $\mathcal{U}(\overline{\Omega})$, the space of restriction to $\Omega$ of functions in $\mathbb{D}(\mathbb{C}^{n})$, is dense in $\widetilde{\mathcal{W}}_{\mathcal{U}}^{\beta,p}(\Omega)$. 
\end{Prop}

\begin{proof}
Let $\mu\in \widetilde{\mathcal{W}}_{\mathcal{U}}^{\beta,p}(\Omega)$ satisfy condition (\ref{eqn2.6}). Let $\displaystyle T:\widetilde{\mathcal{W}}_{\mathcal{U}}^{\beta,p}(\Omega)\longrightarrow\widetilde{\mathcal{W}}_{\mathcal{U}}^{\beta,p}(\mathbb{C})$ be a continuous extension such that $T(\mu)\in\widetilde{\mathcal{W}}_{\mathcal{U}}^{\beta,p}(\mathbb{C})$. Therefore there exists a sequence $\lbrace\varphi_{n}\rbrace$ of ultra-differentiable rapdily decreasing functions in $\mathcal{U}(\mathbb{C}^{n})$ which converges to $T(\mu)$ in $\widetilde{\mathcal{W}}_{\mathcal{U}}^{\beta,p}(\mathbb{C})$. Then the sequence of restriction of the $\varphi_n$ converges to $\mu$ in $\widetilde{\mathcal{W}}_{\mathcal{U}}^{\beta,p}(\Omega)$ for $p>1$ and $0<\beta<1$.
\end{proof}

\begin{Rem}
The next result is as a consequence of Proposition \ref{prop3.8}. In the result, there is an assumption of the complex domain to be a bounded Lipschitz open subset of $\mathbb{C}$ which will give rise to the compactness of embeddings.
\end{Rem}

\begin{Cor}
\label{cor3.9}
Assume that $\beta\in (0,1)$ and $1<p<\infty$, and $\Omega$ a Lipschitz open set in $\mathbb{C}^n$. Then we have:
\begin{itemize}
\item[(i)] if $n>p\beta$, then $\displaystyle\widetilde{\mathcal{W}}_{\mathcal{U}}^{\beta,p}(\Omega)\hookrightarrow L^{p'}(\Omega)$ for every $p'(n-p\beta)\leq np$;
\item[(ii)]if $p\beta=n$, then $\displaystyle\widetilde{\mathcal{W}}_{\mathcal{U}}^{\beta,p}(\Omega)\hookrightarrow L^{p'}(\Omega)$ for any $p'<\infty$;
\item[(iii)]if $p\beta>n$, then $\displaystyle\widetilde{\mathcal{W}}_{\mathcal{U}}^{\beta,p}(\Omega)\hookrightarrow L^{\infty}(\Omega)$.
\end{itemize}
\end{Cor}

\begin{proof}
We present the proof of the result by induction on $[\beta]$ for $n>p\beta$. Assume that the result has been established for $[\beta]=l-1$. Then let $\mu\in \widetilde{\mathcal{W}}_{\mathcal{U}}^{\beta,p}(\Omega)$ with $[\beta]=l$ and $p\beta<n$; then $\nabla \mu\in \widetilde{\mathcal{W}}_{\mathcal{U}}^{\beta-1,p}(\Omega)$, $D^{\alpha}\mu\in \widetilde{\mathcal{W}}_{\mathcal{U}}^{|\alpha|-l,p}(\Omega)$ for $|\alpha|\leq [\beta]=l$, and $\mu\in \widetilde{\mathcal{W}}_{\mathcal{U}}^{[\beta],p}(\Omega)$. Hence, by the induction hypothesis, we have the following: $\nabla\mu\in L^{t}(\Omega)$ with $t=\frac{np}{(n-p(\beta-1))}$ and $\displaystyle\mu\in L^{np((n-p(\beta-1))}(\Omega)\implies \mu\in L^{np/(n-p\beta)}$\ \\
Next, assume that $p\beta=n$, that is, $p\beta-n=0$. Then $[\beta]p<n$ and $(\beta-1)p<n$. Take $\mu\in \widetilde{\mathcal{W}}_{\mathcal{U}}^{\beta,p}(\Omega)$ then we see that $\mu\in \widetilde{\mathcal{W}}_{\mathcal{U}}^{1,t}(\Omega)$ with $t=np/(n-p(\beta-1))=n$. We draw conclusion that $\mu\in L^{p'}$ for any $p'<\infty$ since $t=n$.\ \\
Finally for (iii) we assume that $\beta-1-\frac{n}{p}<r<\beta-\frac{n}{p}$ for any $r\in\mathbb{Z}$ then $\mu\in \widetilde{\mathcal{W}}_{\mathcal{U}}^{\beta,p}(\Omega)$. Then $\upsilon=\nabla^{r}\mu$ belongs to $\widetilde{\mathcal{W}}_{\mathcal{U}}^{\beta-r,p}(\Omega)$. We conclude that for every $\alpha<1$, $D^{r-1}\mu\in C_{K}^{0,\alpha}$ and then $\mu\in C_{K}^{\beta-n/p-1,\alpha}(\Omega)$.
\end{proof}

We now present the following embedding result which is nicely used for the continuous representation of tempered ultradistributions in $\widetilde{\mathcal{W}}_{\mathcal{U}}^{\beta,p}(\Omega)$.

\begin{Th}
\label{theo4.2}
Let $\Omega$ be a bounded domain in $\mathbb{C}$, $\beta\in (0,1)$ and $p\in (1,\infty)$. Assume that $p\beta>1$ and $|\text{Im}(\xi)|<\beta p$ for $\xi\in\Omega$ such that $\displaystyle\sup\Big((1+|\xi|^{p\beta})|\mu(\xi)|^{p}\Big)<\infty$ for $\mu\in \widetilde{\mathcal{W}}_{\mathcal{U}}^{\beta,p}(\Omega)$. If $\mu\in \widetilde{\mathcal{W}}_{\mathcal{U}}^{\beta,p}(\Omega)$, then the injection $\displaystyle \widetilde{\mathcal{W}}_{\mathcal{U}}^{\beta,p}(\Omega)\hookrightarrow C^{\beta-\frac{1}{p}}(\Omega)$ is compact.
\end{Th}

\begin{proof}
The proof of this theorem is similar to the proof presented in Corollary \ref{cor3.9}.
\end{proof}


\begin{thebibliography}{50}
\bibitem{1}Al-Omari, S. K. Q., Al-Omari, J. F. M., On Fourier transforms of ultradistributions of slow growth and their multipliers. Appl. Math. Sc. 4(60),2010, 2963-2970.

\bibitem{2} Alessio F., Raffaella S., Enrico V. Density properties of fractional Sobolev spaces. Annales Academiae Scientiarum Fennicae Mathematica. 40, 2015, 235-253.

\bibitem{3} Amaonyeiro A. U., Egwe M. E. On Tempered ultradistributions in classical Sobolev spaces. arXiv.2046.16912; 2024-arxiv.org preprint.


\bibitem{4} Besov O. V. On a certain family of functional spaces, Imbedding and continuous theorem. Dokl. Akad. Nauk SSSR 126 (1959) 1163-1165.

\bibitem{5} Beurling A. Quasi-analyticity and generalized distributions. Lectures 4 and 5. A.M.S. Summer Institute, Stanford 1961.

\bibitem{6} Belardinelli C. Fractional calculus of tempered distributions. A new approach. arXiv.2105.00276, 2021-arxiv.org preprint.

\bibitem{7} Komatsu H. Ultradistributions I: Structure theorems and a characterization. J. Fac. Sci. Univ. Tokyo, 1973, 10. 25-105.

\bibitem{8} Nenning D. N., Gerhard S. Ultradifferentiable classes of entire functions. Advances in Operator Theory 8. 4(2023). 67.

\bibitem{9} Pathak R. S. Tempered Ultradistribution as Boundary Values of Analytic Functions. Amer.Math.Soc., 537-556.

\bibitem{10} L. Giovanni. A first course in fractional Sobolev spaces. Graduate studies in Mathematics 229. Amer.Math.Soc. 2023.

\bibitem{11}F. Xiabing, M. Sutton. On new family of fractional Sobolev spaces. Banach Jornal of Mathematical Analysis 16. 3(2022): 46.

\bibitem{12}Zemanian A. H. Distribution Theory and Transform analysis. Dover Publications Inc. New York. 1987


\end{thebibliography}
\end{document}